\documentclass[11pt,a4paper]{amsart}
\usepackage[T1]{fontenc}
\usepackage[utf8]{inputenc}
\usepackage{amsmath,amssymb,amsthm}
\usepackage{array}
\usepackage[margin=1in]{geometry}
\usepackage{booktabs}
\usepackage[hidelinks]{hyperref}

\theoremstyle{plain}
\newtheorem{theorem}{Theorem}[section]
\newtheorem{proposition}[theorem]{Proposition}
\newtheorem{lemma}[theorem]{Lemma}
\newtheorem{corollary}[theorem]{Corollary}
\theoremstyle{definition}
\newtheorem{definition}[theorem]{Definition}
\newtheorem{example}[theorem]{Example}
\newtheorem{problem}[theorem]{Problem}
\theoremstyle{remark}
\newtheorem{remark}[theorem]{Remark}

\newcommand{\FF}{\mathbb{F}}
\newcommand{\ZZ}{\mathbb{Z}}
\newcommand{\CC}{\mathbb{C}}
\newcommand{\supp}{\operatorname{supp}}
\newcommand{\tr}{\mathsf{T}}

\begin{document}

\title[Tensor constructions for Euler magic matrices]{Tensor constructions for
Euler magic matrices and proper examples of orders 9, 27, 81 and 243}

\author{Sanjit Singh Mehat}

\date{August 16, 2026}


\begin{abstract}
An \emph{Euler magic matrix} is an integer matrix $M$ satisfying
$MM^{\tr}=\gamma I$ together with two diagonal square-sum conditions; it is
\emph{proper} when its entry squares are pairwise distinct. Müller proved that
Euler magic matrices exist in every order other than $3$, and that no Euler
magic matrix of order $3$ exists at all, while proper examples are considerably
more restrictive. We describe a tensor construction whose factors are only
required to satisfy the orthogonality equation $AA^{\tr}=\gamma I$: for a
linear reindexing $L$ of the row index group $\FF_3^k$ obeying an explicit
support condition, the reindexed Kronecker product of $k$ such $3\times3$
factors satisfies the full Euler magic conditions in order $3^k$. Choosing
factors whose entry squares have pairwise distinct products, we obtain proper
Euler magic matrices of orders $9$, $27$, $81$ and $243$. The construction
therefore produces proper examples in powers of three even though order three
admits no Euler magic matrix, and the passage from the factors to the product
is exactly where the Euler conditions are created rather than inherited. We
give an explicit support condition on $L$ and show that it characterises the
linear reindexings forcing the two Euler diagonal identities for \emph{every}
tuple of semi-magic factor arrays; over $\FF_p$ with two factors, such a
reindexing exists only when $p\le3$. The four existence results are formalised
in Lean~4, as are the two instances of the construction used to obtain them;
the order-$243$ certificate is taken from the archived development and was not
rebuilt in preparing this paper, although its witness was reproduced here by
exact integer arithmetic. Further proper examples of orders $729$ and $2187$ are
verified by exact integer computation only.
\end{abstract}

\maketitle

\section{Introduction and prior work}

Throughout, matrices have integer entries and $n\ge1$. We use the definition of
Müller \cite{Muller2026}.

\begin{definition}[{\cite[Definition~1.1]{Muller2026}}]\label{def:euler}
A matrix $M\in\ZZ^{n\times n}$ is an \emph{Euler magic matrix} with constant
$\gamma\in\ZZ$, $\gamma\ne0$, if
\begin{align}
MM^{\tr} &= \gamma I, \label{eq:orth}\\
\sum_{i=0}^{n-1} M_{i,i}^{2} &= \gamma, \label{eq:main}\\
\sum_{i=0}^{n-1} M_{i,\,n-1-i}^{2} &= \gamma. \label{eq:anti}
\end{align}
It is \emph{proper} if the $n^2$ entry squares $M_{i,j}^2$ are pairwise
distinct.
\end{definition}

Condition \eqref{eq:orth} says that the rows are pairwise orthogonal and each
has square-sum $\gamma$; \eqref{eq:main} and \eqref{eq:anti} impose the same
square-sum on the two diagonals. An Euler magic matrix is thus a square array
of integers whose rows, columns and both diagonals have equal square-sums,
together with the orthogonality of distinct rows. Properness is the natural
non-degeneracy requirement, and is what makes the problem an analogue of the
classical question of magic squares of squares.

\subsection*{The known landscape}

The state of the art we take as our starting point is the following.

\begin{itemize}
\item Euler gave a proper example of order $4$; it is the classical
  four-square identity matrix and is reproduced in \cite[\S1]{Muller2026}.
\item Müller \cite[Theorem~1.2]{Muller2026} proved that \emph{no} Euler magic
  matrix of order $3$ exists --- not merely that no proper one does. He also
  exhibited a proper example of order $8$.
\item Müller \cite[Theorem~4.1]{Muller2026} proved that an Euler magic matrix
  exists in \emph{every} order $n\ne3$. Ordinary existence is therefore settled,
  including at every power of three; the residual difficulty is entirely in
  properness.
\item Kominers \cite{Kominers2026} gave a proper example of order $5$ and
  observed that beyond order $5$ the existence of proper Euler magic matrices
  in odd orders was open, suggesting a product approach without carrying it
  out.
\end{itemize}

A proper Euler magic matrix of order $6$ is obtained by different means in
separate, contemporaneous work of the author; it is not used anywhere below and
is not part of the construction of this paper.

\subsection*{Adjacent product constructions}\label{sec:adjacent}

Kronecker products are a standard tool for building magic squares, and there is
a substantial literature on \emph{compound} magic squares --- including at
orders $3^{\ell}$, which is exactly the range treated here. It is therefore
worth saying precisely what separates that literature from
Definition~\ref{def:euler}.

The distinguishing condition is \eqref{eq:orth}. A magic square requires equal
\emph{line sums}; a multimagic square requires equal line sums of the first
several powers of the entries; a magic square of squares requires equal line
sums of the entry squares. None of these implies that distinct rows are
orthogonal, which is what \eqref{eq:orth} demands, and orthogonality is not
preserved by the operations (relabelling, affine transformation of entries)
under which those classes are closed. Conversely \eqref{eq:orth} does not
require the entries themselves to have equal line sums. The two notions are
independent, and a construction producing one says nothing about the other.
Concretely:

\begin{itemize}
\item Loly and Cameron \cite{LolyCameron2020} extend Frierson's 1907
  parametrisation of compound magic squares to all orders $3^{\ell}$, and
  Rogers, Cameron and Loly \cite{RogersCameronLoly2017} compound
  \emph{doubly affine} matrices --- arrays whose defining property is, in their
  words, ``only identical row and column sums (often called semi-magic)''. That
  is precisely the hypothesis of our Lemma~\ref{lem:semimagic}, but it is applied
  there to the entries, and the conclusion drawn is about line sums and spectra
  of the compound, not about $MM^{\tr}$. Neither paper imposes or derives
  \eqref{eq:orth}, and neither addresses the two diagonal square-sum conditions
  \eqref{eq:main}--\eqref{eq:anti} or properness in the sense of
  Definition~\ref{def:euler}.
\item Nordgren \cite{Nordgren2024} constructs compound magic square matrices
  from Lucas sequences, again at composite orders via Kronecker products; the
  invariants studied are ranks, eigenvalues and line sums.
\item Zhang, Chen and Li \cite{ZhangChenLi2015} give product constructions for
  multimagic squares using diagonal Latin squares and Kronecker products, and
  similar order-$9$ constructions arise from pairs of orthogonal Latin squares of
  order $3$. Multimagic conditions constrain sums of powers along lines; they do
  not give row orthogonality. Kominers cites this same work
  \cite[\S5]{Kominers2026} when noting that keeping a Kronecker product proper
  requires care.
\item Rome and Yamagishi \cite{RomeYamagishi2025} prove by the circle method
  that magic squares of powers exist in every order $n\ge4$. As Kominers observes
  \cite[\S1]{Kominers2026}, a magic square of squares need not arise from an
  orthogonal matrix, so this does not produce Euler magic matrices;
  \eqref{eq:orth} is the extra rigidity.
\item Pirsic \cite{Pirsic2019,Pirsic2020} parametrises $8\times8$ magic squares
  of squares through octonionic multiplication, and relates magic squares of
  squares to division algebras. This is the closest in spirit --- our
  Lemma~\ref{lem:quaternion} is the quaternionic analogue at order $3$ --- but the
  orders reached are $4$ and $8$, governed by the associative and alternative
  division algebras, and no power of three arises.
\end{itemize}

We have not located a source that implies Theorem~\ref{thm:criterion} or the
existence of proper Euler magic matrices in orders $9$, $27$, $81$ or $243$. We
state that as the outcome of a search, not as a proof of novelty: absence of a
located source is not evidence of absence, and the literature on magic squares
is old, large and scattered across several communities.

We stress the third point: \emph{we do not claim ordinary existence at powers of
three}, which is Müller's Theorem~4.1. Our focus here is properness and the
mechanism by which the Euler conditions are produced.

\subsection*{The problem}

Since no Euler magic matrix of order $3$ exists, a product construction cannot
build order $9$ --- or any order $3^k$ --- out of Euler magic factors of order
$3$: there are none. Thus a product construction requiring order-$3$ Euler
magic factors cannot even begin. The question we address is whether the
obstruction in order $3$ propagates:

\begin{quote}
Do proper Euler magic matrices exist in orders that are powers of three, given
that order three admits no Euler magic matrix at all?
\end{quote}

\subsection*{Results}

The answer is affirmative in the first four cases.

\begin{quote}
\emph{Proper Euler magic matrices exist in orders $9$, $27$, $81$ and $243$.}
\end{quote}

The mechanism is that the factors need not be Euler magic. Only the
orthogonality equation $AA^{\tr}=\gamma I$ is imposed on them; a particular
linear reindexing of the row index group of the tensor product then supplies
the two diagonal conditions, which the plain Kronecker product does not satisfy.
Order $9$ is the smallest power of three settled by the $\FF_3$ construction used
here, and is where the phenomenon is clearest: it is a power of three, its only
nontrivial factorisation is $3\times3$, and order $3$ admits no Euler magic
matrix, so the order-$9$ example cannot be a product of Euler magic matrices.

Section~\ref{sec:construction} isolates the reindexings that work. The condition
is a support condition on the transpose of the reindexing map
(Theorem~\ref{thm:criterion}), proved by a character computation, and it is not
merely sufficient: it characterises those $L$ that force the identities for
every tuple of semi-magic factor arrays (Theorem~\ref{thm:converse}). For a
two-factor product over $\FF_p$ such an $L$ exists only when $p\le3$
(Corollary~\ref{cor:whythree}), which is what singles out $p=3$ for the base
construction used here.

Sections~\ref{sec:nine} and~\ref{sec:tower} give the witnesses,
Section~\ref{sec:proper} treats properness, Section~\ref{sec:computational}
records exactly-computed extensions to orders $729$ and $2187$,
Section~\ref{sec:verification} states the verification status of each claim,
and Section~\ref{sec:open} states the open problems the construction generates.

\section{Euler magic matrices and orthogonality factors}\label{sec:prelim}

\begin{definition}\label{def:factor}
An \emph{orthogonality factor} with constant $\gamma$ is a matrix
$A\in\ZZ^{n\times n}$ with $AA^{\tr}=\gamma I$ and $\gamma\ne0$. We write
$a$ for its \emph{entry-square array}, $a(x,y)=A_{x,y}^2$.
\end{definition}

An Euler magic matrix is an orthogonality factor satisfying two further
conditions; the converse fails, and the failure is what the construction
exploits.

\begin{lemma}\label{lem:semimagic}
Let $A\in\ZZ^{n\times n}$ be an orthogonality factor with constant $\gamma$.
Then every row sum and every column sum of the entry-square array $a$ equals
$\gamma$:
\[
\sum_{y} a(x,y)=\gamma \quad (0\le x<n),\qquad
\sum_{x} a(x,y)=\gamma \quad (0\le y<n).
\]
\end{lemma}

\begin{proof}
The diagonal entries of $AA^{\tr}=\gamma I$ give
$\sum_y A_{x,y}^2=\gamma$ for each $x$, which is the row statement. For the
columns, note that $AA^{\tr}=\gamma I$ with $\gamma\ne0$ forces $A$ to be
invertible over $\mathbb{Q}$ with $A^{-1}=\gamma^{-1}A^{\tr}$, whence
$A^{\tr}A=\gamma I$ as well; its diagonal entries give
$\sum_x A_{x,y}^2=\gamma$.
\end{proof}

Lemma~\ref{lem:semimagic} is classical and goes back to Euler; it is stated in
this form, with the same proof, by Kominers \cite[\S1]{Kominers2026}, who
records that once $MM^{\tr}=\gamma I$ holds ``only the two diagonal sums and the
distinctness of the entry-squares remain''. We include it because the semi-magic
property of the entry-square array is the exact hypothesis of the results in
Section~\ref{sec:construction}.

We call an array with all row sums and all column sums equal to $\gamma$
\emph{semi-magic with line sum $\gamma$}. Lemma~\ref{lem:semimagic} says that
the entry-square array of an orthogonality factor is semi-magic; note that
nothing is asserted about its diagonals, and in general the diagonal sums differ
from $\gamma$. Conditions \eqref{eq:main} and \eqref{eq:anti} are exactly the
statement that the entry-square array is moreover \emph{magic}.

\begin{example}\label{ex:factors}
The two matrices
\[
A=\begin{pmatrix}28&47&16\\44&-32&17\\23&4&-52\end{pmatrix},
\qquad
B=\begin{pmatrix}75&54&22\\50&-78&21\\30&-5&-90\end{pmatrix}
\]
are orthogonality factors with constants $\gamma_A=3249=57^2$ and
$\gamma_B=9025=95^2$, and each has nine pairwise distinct entry squares.
Neither is an Euler magic matrix: their main-diagonal square-sums are
\[
28^2+(-32)^2+(-52)^2=4512\ne3249,\qquad
75^2+(-78)^2+(-90)^2=19809\ne9025 .
\]
Indeed no $3\times3$ matrix whatever is an Euler magic matrix, by
\cite[Theorem~1.2]{Muller2026}. Orthogonality factors of order $3$, by
contrast, are plentiful, as the next lemma records.
\end{example}

\begin{lemma}\label{lem:quaternion}
For integers $w,x,y,z$ put $N=w^2+x^2+y^2+z^2$ and
\begin{equation}\label{eq:quat}
R(w,x,y,z)=\begin{pmatrix}
w^2{+}x^2{-}y^2{-}z^2 & 2(xy-wz) & 2(xz+wy)\\
2(xy+wz) & w^2{-}x^2{+}y^2{-}z^2 & 2(yz-wx)\\
2(xz-wy) & 2(yz+wx) & w^2{-}x^2{-}y^2{+}z^2
\end{pmatrix}.
\end{equation}
Then $R R^{\tr}=N^2 I$, so $R$ is an orthogonality factor with constant $N^2$
whenever $N\ne0$. Conversely, every orthogonality factor $A\in\ZZ^{3\times3}$
has square constant: $\gamma$ is a perfect square.
\end{lemma}

The converse half is the case $n=3$ of a lemma of Müller
\cite[Lemma~2.1]{Muller2026}, who proves it for every odd $n$ and deduces that
$A/\sqrt{\gamma}$ is a rational orthogonal matrix; see also
Kominers \cite[Lemma~2.1]{Kominers2026}. We give the short argument for
completeness.

\begin{proof}
The identity $RR^{\tr}=N^2I$ is the classical quaternion rotation formula and is
checked directly. For the converse, taking determinants in $AA^{\tr}=\gamma I$
gives $(\det A)^2=\gamma^3$. Hence $\gamma>0$, and in the prime factorisation of
$\gamma$ every exponent $e$ satisfies $3e$ even, so $e$ is even and $\gamma$ is a
square.
\end{proof}

Transposition, permutations of rows or columns, and sign changes of whole rows
or columns all send orthogonality factors to orthogonality factors with the same
constant. Modulo these operations, the factors $A$ and $B$ of
Example~\ref{ex:factors} arise from \eqref{eq:quat} at the quaternions
$(-6,-4,-2,-1)$ and $(-7,-6,-3,-1)$, of norms $57$ and $95$; their constants are
accordingly $57^2=3249$ and $95^2=9025$. Every constant in
Table~\ref{tab:tower} is a perfect square, as Lemma~\ref{lem:quaternion}
requires.

Example~\ref{ex:factors} is the point of departure. The class of order-$3$
orthogonality factors is nonempty and rich, while the class of order-$3$ Euler
magic matrices is empty. A construction that consumes only the former can
therefore reach orders divisible by $3$, whereas one that consumes the latter
cannot get started.

\section{The lexicographically reindexed tensor construction}
\label{sec:construction}

\subsection{The construction}

Identify $\{0,1,\dots,3^k-1\}$ with $\FF_3^k$ by base-three digits, most
significant first: the index $v=(v_1,\dots,v_k)\in\FF_3^k$ corresponds to
$\sum_{i=1}^{k} v_i\,3^{\,k-i}$. Under this identification the plain Kronecker
product $A_1\otimes\cdots\otimes A_k$ has entries
$\prod_{i} (A_i)_{v_i,w_i}$ at position $(v,w)$; this is the lexicographic
ordering of the tensor index.

\begin{definition}\label{def:kronlex}
Let $k\ge1$, let $L\in GL_k(\FF_3)$, and let $A_1,\dots,A_k\in\ZZ^{3\times3}$.
The \emph{$L$-reindexed tensor product} is the matrix
$M_L(A_1,\dots,A_k)\in\ZZ^{3^k\times3^k}$ with entries
\[
M_L(A_1,\dots,A_k)_{v,w}\;=\;\prod_{i=1}^{k}\,(A_i)_{(Lv)_i,\;w_i},
\qquad v,w\in\FF_3^k .
\]
\end{definition}

Thus the columns are indexed lexicographically and the rows are indexed
lexicographically after applying $L$; equivalently $M_L$ is the row permutation
of $A_1\otimes\cdots\otimes A_k$ induced by $L$. Taking $L=I$ recovers the plain
Kronecker product.

\subsection{Orthogonality is inherited; the diagonals are not}

\begin{lemma}\label{lem:orth}
With the notation of Definition~\ref{def:kronlex}, if each $A_i$ is an
orthogonality factor with constant $\gamma_i$ and $\gamma=\prod_i\gamma_i\ne0$,
then $M_LM_L^{\tr}=\gamma I$ for every $L\in GL_k(\FF_3)$.
\end{lemma}

\begin{proof}
The Kronecker product satisfies
$(A_1\otimes\cdots\otimes A_k)(A_1\otimes\cdots\otimes A_k)^{\tr}
=\bigotimes_i (A_iA_i^{\tr})=\bigotimes_i(\gamma_i I)=\gamma I$.
Reindexing rows by a bijection $\sigma$ and columns by a bijection $\tau$
replaces $N$ by $P_\sigma N P_\tau^{\tr}$ for permutation matrices
$P_\sigma,P_\tau$; then
$(P_\sigma NP_\tau^{\tr})(P_\sigma NP_\tau^{\tr})^{\tr}
=P_\sigma NN^{\tr}P_\sigma^{\tr}=\gamma P_\sigma P_\sigma^{\tr}=\gamma I$.
\end{proof}

Conditions \eqref{eq:main} and \eqref{eq:anti} behave quite differently: they
refer to specific positions and are destroyed by reindexing. In particular they
are \emph{not} automatic for the plain Kronecker product.

\begin{example}[the reindexing is necessary]\label{ex:necessary}
For the factors $A,B$ of Example~\ref{ex:factors}, the plain Kronecker product
$A\otimes B$ satisfies $(A\otimes B)(A\otimes B)^{\tr}=\gamma I$ with
$\gamma=\gamma_A\gamma_B=29\,322\,225$ by Lemma~\ref{lem:orth}, and its $81$
entry squares are pairwise distinct. But its diagonal square-sums are
\[
\sum_{i=0}^{8}(A\otimes B)_{i,i}^2=89\,378\,208,
\qquad
\sum_{i=0}^{8}(A\otimes B)_{i,8-i}^2=13\,509\,612,
\]
neither equal to $\gamma$. So $A\otimes B$ is not an Euler magic matrix, and no
appeal to multiplicativity of the Kronecker product can produce one. The
diagonal conditions must be \emph{created}, and that is what the reindexing does.
\end{example}

\subsection{Which reindexings work}

For $s\in\FF_3^k$ write $\supp(s)=\{\,i : s_i\ne0\,\}$.

\begin{theorem}\label{thm:criterion}
Let $k\ge1$ and $L\in GL_k(\FF_3)$, and suppose
\begin{equation}\tag{$\star$}\label{eq:star}
\supp(L^{\tr}s)\ne\supp(s)\qquad\text{for every } s\in\FF_3^k\setminus\{0\}.
\end{equation}
Let $A_1,\dots,A_k\in\ZZ^{3\times3}$ be orthogonality factors with constants
$\gamma_1,\dots,\gamma_k$ and put $\gamma=\prod_i\gamma_i\ne0$. Then
$M=M_L(A_1,\dots,A_k)$ satisfies
\[
\sum_{v\in\FF_3^k} M_{v,v}^{2}=\gamma
\qquad\text{and}\qquad
\sum_{v\in\FF_3^k} M_{v,\bar v}^{2}=\gamma ,
\]
where $\bar v=(2,\dots,2)-v$ is the antidiagonal partner of $v$. Consequently
$M$ is an Euler magic matrix of order $3^k$ with constant $\gamma$.
\end{theorem}

\begin{proof}
Let $\omega=e^{2\pi i/3}$ and for an array $a:\FF_3\times\FF_3\to\ZZ$ write
\[
\widehat a(s,t)=\tfrac19\sum_{x,y\in\FF_3}a(x,y)\,\omega^{-sx-ty},
\qquad
a(x,y)=\sum_{s,t\in\FF_3}\widehat a(s,t)\,\omega^{sx+ty}.
\]
Let $a_i$ be the entry-square array of $A_i$. By Lemma~\ref{lem:semimagic},
$a_i$ is semi-magic with line sum $\gamma_i$. We first record what this says on
the Fourier side. For fixed $x$,
\[
\sum_{y}a_i(x,y)=\sum_{s,t}\widehat{a_i}(s,t)\,\omega^{sx}\sum_{y}\omega^{ty}
=3\sum_{s}\widehat{a_i}(s,0)\,\omega^{sx},
\]
since $\sum_y\omega^{ty}=3$ if $t=0$ and $0$ otherwise. This equals $\gamma_i$
for all $x$ exactly when $\widehat{a_i}(0,0)=\gamma_i/3$ and
$\widehat{a_i}(s,0)=0$ for $s\ne0$. Symmetrically, equality of all column sums
gives $\widehat{a_i}(0,t)=0$ for $t\ne0$. Hence
\begin{equation}\label{eq:support}
\widehat{a_i}(s,t)=0 \quad\text{unless } (s,t)=(0,0) \text{ or } s\ne0\ne t,
\qquad \widehat{a_i}(0,0)=\gamma_i/3 .
\end{equation}

Now compute the main-diagonal square-sum. By
Definition~\ref{def:kronlex}, $M_{v,v}^2=\prod_i a_i((Lv)_i,v_i)$, so
\begin{align*}
D:=\sum_{v\in\FF_3^k}M_{v,v}^2
&=\sum_{v}\prod_{i=1}^{k}\ \sum_{s_i,t_i\in\FF_3}
   \widehat{a_i}(s_i,t_i)\,\omega^{\,s_i(Lv)_i+t_iv_i}\\
&=\sum_{s,t\in\FF_3^k}\Big(\prod_i\widehat{a_i}(s_i,t_i)\Big)
   \sum_{v}\omega^{\,\langle s,Lv\rangle+\langle t,v\rangle}\\
&=\sum_{s,t}\Big(\prod_i\widehat{a_i}(s_i,t_i)\Big)
   \sum_{v}\omega^{\,\langle L^{\tr}s+t,\;v\rangle}
\;=\;3^k\sum_{s\in\FF_3^k}\ \prod_{i=1}^{k}
   \widehat{a_i}\big(s_i,\,-(L^{\tr}s)_i\big),
\end{align*}
using $\sum_{v\in\FF_3^k}\omega^{\langle u,v\rangle}=3^k$ if $u=0$ and $0$
otherwise, which forces $t=-L^{\tr}s$.

Fix $s\ne0$. By \eqref{eq:support} the $i$-th factor
$\widehat{a_i}(s_i,-(L^{\tr}s)_i)$ vanishes unless $s_i$ and $(L^{\tr}s)_i$ are
either both zero or both nonzero. So the whole product vanishes unless
$\supp(L^{\tr}s)=\supp(s)$, which \eqref{eq:star} excludes. Only $s=0$
contributes, and there $L^{\tr}s=0$, giving
\[
D=3^k\prod_{i=1}^{k}\widehat{a_i}(0,0)=3^k\prod_{i=1}^{k}\frac{\gamma_i}{3}
=\prod_{i=1}^{k}\gamma_i=\gamma .
\]

For the antidiagonal, $\bar v_i=2-v_i$, so
$M_{v,\bar v}^2=\prod_i a_i((Lv)_i,2-v_i)$ and the same expansion gives
\[
D':=\sum_{v}M_{v,\bar v}^2
=\sum_{s,t}\Big(\prod_i\widehat{a_i}(s_i,t_i)\Big)\,\omega^{2\sum_i t_i}
  \sum_{v}\omega^{\langle L^{\tr}s-t,\,v\rangle}
=3^k\sum_{s}\omega^{2\sum_i (L^{\tr}s)_i}\prod_{i}
  \widehat{a_i}\big(s_i,(L^{\tr}s)_i\big).
\]
The vanishing analysis is identical --- the extra root of unity is a scalar and
does not affect which terms survive --- so again only $s=0$ contributes and
$D'=\gamma$. Together with Lemma~\ref{lem:orth} this gives all three conditions
of Definition~\ref{def:euler}.
\end{proof}

Condition \eqref{eq:star} is not merely a convenient sufficient condition: it is
exactly the condition under which a linear reindexing forces the two identities
\emph{universally}, that is for every tuple of semi-magic factor arrays. We
state this for a general prime, since nothing in the argument is special to $3$.

\begin{theorem}[universal linear-reindexing criterion]\label{thm:converse}
Let $p$ be prime, $k\ge1$ and $L\in GL_k(\FF_p)$. Let $W$ be the set of
$p\times p$ integer arrays all of whose row sums and all of whose column sums
are equal to a common value $\gamma(a)$. The following are equivalent.
\begin{enumerate}
\item[\textup{(i)}] $\supp(L^{\tr}s)\ne\supp(s)$ for every
  $s\in\FF_p^k\setminus\{0\}$.
\item[\textup{(ii)}] For \emph{every} $a_1,\dots,a_k\in W$,
\[
\sum_{v\in\FF_p^k}\prod_{i=1}^{k}a_i\big((Lv)_i,\,v_i\big)
\;=\;
\sum_{v\in\FF_p^k}\prod_{i=1}^{k}a_i\big((Lv)_i,\,p-1-v_i\big)
\;=\;\prod_{i=1}^{k}\gamma(a_i).
\]
\end{enumerate}
Each of the two identities in \textup{(ii)}, taken alone, is already equivalent
to \textup{(i)}.
\end{theorem}

\begin{proof}
Let $F$ and $F'$ be the two differences (left side minus $\prod_i\gamma(a_i)$)
in (ii), regarded as functions of $(a_1,\dots,a_k)$. Both are multilinear: the
sums are multilinear by inspection, and $\prod_i\gamma(a_i)$ is a product of
linear functionals in distinct arguments.

$W$ is a free $\ZZ$-module of rank $d=(p-1)^2+1$, with coordinates the entries
$a_{x,y}$ for $x,y<p-1$ together with $\gamma(a)$: the remaining entries are
recovered by $a_{x,p-1}=\gamma-\sum_{y<p-1}a_{x,y}$,
$a_{p-1,y}=\gamma-\sum_{x<p-1}a_{x,y}$ and
$a_{p-1,p-1}=\gamma-\sum_{y<p-1}a_{p-1,y}$, and one checks that the last column
then sums to $\gamma$ automatically. The complex arrays satisfying the same
conditions form a space $W_\CC$ of dimension $d$ over $\CC$, cut out by the same
integer linear equations, so a $\ZZ$-basis of $W$ is a $\CC$-basis of $W_\CC$.
A multilinear form is determined by its values on tuples of basis elements, and
those values are the same numbers whether $F$ is read on $W^k$ or on $W_\CC^k$.
Hence $F$ vanishes on $W^k$ if and only if its $\CC$-multilinear extension
vanishes on $W_\CC^k$; in particular we may test on any $\CC$-basis of $W_\CC$,
and no question of rationality, or of the reality of an individual array, arises.

Let $\omega=e^{2\pi i/p}$, $\chi_{(s,t)}(x,y)=\omega^{sx+ty}$, and
$B=\{(0,0)\}\cup\{(s,t):s\ne0\ne t\}$. For $t\ne0$ every row sum of
$\chi_{(s,t)}$ vanishes and for $s\ne0$ every column sum vanishes, so
$\chi_{(s,t)}\in W_\CC$ for $(s,t)\in B$, while $\chi_{(0,0)}$ is the all-ones
array. Distinct characters of $(\ZZ/p)^2$ are linearly independent and
$|B|=(p-1)^2+1=d$, so these form a $\CC$-basis of $W_\CC$. Moreover
$\gamma(\chi_{(0,0)})=p$ and $\gamma(\chi_{(s,t)})=0$ for $s,t\ne0$.

Evaluate at $a_i=\chi_{(\sigma_i,\tau_i)}$ with $(\sigma_i,\tau_i)\in B$. Since
$\sum_{v\in\FF_p^k}\omega^{\langle u,v\rangle}$ is $p^k$ for $u=0$ and $0$
otherwise,
\[
\sum_{v}\prod_i\omega^{\sigma_i(Lv)_i+\tau_iv_i}
=\sum_{v}\omega^{\langle L^{\tr}\sigma+\tau,\;v\rangle}
=p^k\,\big[\,L^{\tr}\sigma+\tau=0\,\big],
\]
whereas $\prod_i\gamma(\chi_{(\sigma_i,\tau_i)})$ is $p^k$ when every
$(\sigma_i,\tau_i)=(0,0)$ and $0$ otherwise. In that first case
$\sigma=\tau=0$, both sides equal $p^k$ and $F$ vanishes. Otherwise the second
quantity is $0$, so $F$ is nonzero at this basis tuple precisely when
$\tau=-L^{\tr}\sigma$. Such a tuple exists with $(\sigma,\tau)\ne0$ exactly when
some $\sigma\ne0$ has, for every $i$, $\sigma_i$ and $(L^{\tr}\sigma)_i$ both
zero or both nonzero --- that is, $\supp(L^{\tr}\sigma)=\supp(\sigma)$. So
$F\equiv0$ if and only if (i) holds. For $F'$ the same computation acquires the
factor $\omega^{(p-1)\sum_i\tau_i}$, a root of unity and so nonzero, and the
condition $\tau=L^{\tr}\sigma$; the support condition is unchanged. Hence
$F'\equiv0$ if and only if (i) holds as well.
\end{proof}

\begin{remark}\label{rem:sharp}
The domain in (ii) matters and should not be enlarged in the reading.
Theorem~\ref{thm:converse} characterises the $L$ that force the identities for
\emph{every} semi-magic tuple. By Lemma~\ref{lem:semimagic} the entry-square
arrays of orthogonality factors are semi-magic, but they do not exhaust $W$:
they are nonnegative and subject to further arithmetic constraints. So an $L$
violating \eqref{eq:star} is not thereby prevented from satisfying the
identities on some particular, specially structured family of factors. We make
no claim in that direction, and nothing below depends on one.

As an independent check on Theorem~\ref{thm:converse}, both identities were
evaluated on all $5^k$ tuples drawn from a basis of $W$ --- which by
multilinearity decides them for all semi-magic arrays --- for every $L$ in
$GL_2(\FF_3)$ and $GL_3(\FF_3)$. Condition \eqref{eq:star} held for exactly
those $L$ for which the identities held, in all $48+11\,232$ cases; $8$ and
$480$ maps respectively satisfy \eqref{eq:star}. This is verification, not part
of the proof.
\end{remark}

\subsection{Why powers of three}

Definition~\ref{def:kronlex} and condition \eqref{eq:star} make sense verbatim
over $\FF_p$ for any prime $p$, building order $p^k$ from $p\times p$
orthogonality factors; Lemma~\ref{lem:semimagic} and the character computation
go through unchanged. The following shows that for a two-factor product the
condition is satisfiable only for $p\le3$. This is what confines the base
blocks of our construction to $p=3$; see Remark~\ref{rem:scope} for what it
does and does not exclude.

\begin{proposition}\label{prop:whythree}
Let $p$ be prime. There exists $L\in GL_2(\FF_p)$ satisfying \eqref{eq:star} if
and only if $p\le3$.
\end{proposition}

\begin{proof}
Write $L=\begin{pmatrix}l_{11}&l_{12}\\l_{21}&l_{22}\end{pmatrix}$, so that
$(L^{\tr}s)_1=l_{11}s_1+l_{21}s_2$ and $(L^{\tr}s)_2=l_{12}s_1+l_{22}s_2$.

Suppose $L$ satisfies \eqref{eq:star}, and consider $s=(a,b)$ with
$a,b\ne0$, so $\supp(s)=\{1,2\}$. Then \eqref{eq:star} requires
$\supp(L^{\tr}s)\ne\{1,2\}$, that is
\[
l_{11}a+l_{21}b=0 \qquad\text{or}\qquad l_{12}a+l_{22}b=0 .
\]
Dividing by $a$ and writing $r=b/a$, which ranges over all of $\FF_p^{\times}$
as $(a,b)$ ranges over pairs of nonzero elements, we need
\[
l_{11}+l_{21}r=0 \qquad\text{or}\qquad l_{12}+l_{22}r=0
\qquad\text{for every } r\in\FF_p^{\times}.
\]
Neither $(l_{11},l_{21})$ nor $(l_{12},l_{22})$ is the zero vector, since $L$ is
invertible and these are its columns. Hence each of the two displayed equations
is a nontrivial linear equation in $r$ and has at most one root in $\FF_p$.
Their roots therefore cover at most two elements of $\FF_p^{\times}$, so
$p-1\le2$, i.e.\ $p\le3$.

Conversely, for $p=3$ the matrix $L=\begin{pmatrix}1&1\\2&1\end{pmatrix}$
satisfies \eqref{eq:star}: its transpose sends $(1,0)\mapsto(1,1)$,
$(0,1)\mapsto(2,1)$, $(1,1)\mapsto(0,2)$ and $(1,2)\mapsto(2,0)$, and in each
case the support changes; the remaining nonzero $s$ are scalar multiples of
these, with the same supports. For $p=2$ the condition is satisfied by
$\begin{pmatrix}1&1\\1&0\end{pmatrix}$, giving order $4$, where a proper
example is classical.
\end{proof}

Combining with Theorem~\ref{thm:converse} turns this into a statement about the
mechanism rather than about one criterion.

\begin{corollary}\label{cor:whythree}
Let $p$ be prime. There exists $L\in GL_2(\FF_p)$ for which the two diagonal
identities hold for \emph{every} pair of semi-magic $p\times p$ arrays if and
only if $p\le3$.
\end{corollary}

\begin{proof}
Immediate from Proposition~\ref{prop:whythree} and
Theorem~\ref{thm:converse} applied with $k=2$.
\end{proof}

\begin{remark}\label{rem:scope}
Corollary~\ref{cor:whythree} should be read narrowly, and in particular its
universal quantifier should not be dropped. It says that for $p>3$ no linear
reindexing of a two-factor product forces the diagonal identities for
\emph{every} pair of semi-magic factors. It does \emph{not} say that no linear
reindexing can ever succeed for $p>3$: an $L$ violating \eqref{eq:star} may
still satisfy the identities on particular, specially structured factors, and
the entry-square arrays of orthogonality factors are exactly such a restricted
family (Remark~\ref{rem:sharp}). We have not investigated that question.

Three further things are not excluded: some $L\in GL_k(\FF_p)$ satisfying
\eqref{eq:star} for $p>3$ and $k\ge3$, which we have not investigated; affine
or nonlinear reindexings; and any other route to proper examples in orders
$p^k$. What is established is that the universal two-factor mechanism used here
to build the base blocks singles out $p=3$, which is why the present paper is
about powers of three.
\end{remark}

\begin{remark}[the case $k=2$ is classical as an index structure]\label{rem:mols}
Write $L=\begin{pmatrix}a&b\\c&d\end{pmatrix}\in GL_2(\FF_3)$ and let
$S_1(v_1,v_2)=av_1+bv_2$ and $S_2(v_1,v_2)=cv_1+dv_2$ be the two coordinate
functions of $L$. Then $S_1$ is a Latin square of order $3$ exactly when
$a,b\ne0$, similarly for $S_2$, and $S_1,S_2$ are orthogonal exactly when $L$ is
invertible. So $(S_1,S_2)$ is a pair of orthogonal Latin squares precisely when
$L$ is invertible with all four entries nonzero, and a direct check of the $48$
elements of $GL_2(\FF_3)$ shows that these are exactly the $8$ maps satisfying
\eqref{eq:star}. For $L_9$ the pair is the classical order-$3$ example
\[
S_1=\begin{pmatrix}0&1&2\\1&2&0\\2&0&1\end{pmatrix},\qquad
S_2=\begin{pmatrix}0&1&2\\2&0&1\\1&2&0\end{pmatrix}.
\]
Orthogonal Latin squares of order $3$ and their Kronecker products are of course
long established, and appear in the product constructions for multimagic squares
of \cite{ZhangChenLi2015}. What does not follow from that structure is the
content of Theorem~\ref{thm:criterion}: orthogonality of the Latin pair says
nothing about the two \emph{diagonal square-sum} conditions of
Definition~\ref{def:euler}, which is what \eqref{eq:star} delivers. We do not
pursue the analogous description for $k\ge3$.
\end{remark}

\subsection{The two reindexings used below}

\begin{corollary}\label{cor:LL}
The matrices
\[
L_9=\begin{pmatrix}1&1\\2&1\end{pmatrix}\in GL_2(\FF_3),
\qquad
L_{27}=\begin{pmatrix}0&0&1\\1&1&0\\1&2&0\end{pmatrix}\in GL_3(\FF_3)
\]
satisfy \eqref{eq:star}. Consequently, for orthogonality factors
$A,B,C\in\ZZ^{3\times3}$ with constants $\gamma_A,\gamma_B,\gamma_C$,
\[
M_{L_9}(A,B)\ \text{is Euler magic of order }9\text{ with constant }
\gamma_A\gamma_B,
\]
\[
M_{L_{27}}(A,B,C)\ \text{is Euler magic of order }27\text{ with constant }
\gamma_A\gamma_B\gamma_C,
\]
provided the respective constants are nonzero.
\end{corollary}

\subsection{Direct sums of reindexings}

Reindexings compose in the obvious way, and this is how the higher orders below
are assembled. For $L\in GL_k(\FF_3)$ and $L'\in GL_m(\FF_3)$ write
$L\oplus L'\in GL_{k+m}(\FF_3)$ for the block-diagonal map.

\begin{lemma}\label{lem:dsum}
Let $L\in GL_k(\FF_3)$, $L'\in GL_m(\FF_3)$, and let
$A_1,\dots,A_k,B_1,\dots,B_m\in\ZZ^{3\times3}$. Then
\begin{equation}\label{eq:dsum}
M_{L\oplus L'}(A_1,\dots,A_k,B_1,\dots,B_m)
= M_L(A_1,\dots,A_k)\otimes M_{L'}(B_1,\dots,B_m),
\end{equation}
the Kronecker product being taken in the lexicographic identification.
Moreover $L\oplus L'$ satisfies \eqref{eq:star} if and only if both $L$ and
$L'$ do.
\end{lemma}

\begin{proof}
Write an index of $\FF_3^{k+m}$ as $v=(x,y)$ with $x\in\FF_3^k$, $y\in\FF_3^m$;
this is exactly the lexicographic identification of $\{0,\dots,3^{k+m}-1\}$ with
the product. Since $(L\oplus L')(x,y)=(Lx,L'y)$, the entry of the left-hand side
of \eqref{eq:dsum} at $((x,y),(u,z))$ is
\[
\prod_{i=1}^{k}(A_i)_{(Lx)_i,u_i}\ \prod_{j=1}^{m}(B_j)_{(L'y)_j,z_j},
\]
which is the corresponding entry of the right-hand side.

For the second statement, let $s=(s',s'')$ be nonzero. Then
$(L\oplus L')^{\tr}s=(L^{\tr}s',L'^{\tr}s'')$, so
$\supp((L\oplus L')^{\tr}s)=\supp(s)$ holds if and only if
$\supp(L^{\tr}s')=\supp(s')$ and $\supp(L'^{\tr}s'')=\supp(s'')$. Taking
$s''=0$, for which the second condition is automatic, shows that a failure of
\eqref{eq:star} for $L$ produces one for $L\oplus L'$, and symmetrically for
$L'$; conversely a failure for $L\oplus L'$ at some $s\ne0$ forces a failure for
$L$ or for $L'$ at $s'$ or $s''$, whichever is nonzero.
\end{proof}

Finally, once Euler magic matrices are available, the plain lexicographic
Kronecker product preserves the conditions. This is due to
Kominers \cite[\S5]{Kominers2026}; we include the short proof for completeness.

\begin{lemma}[Kominers]\label{lem:compose}
If $M$ and $N$ are Euler magic matrices of orders $m,n$ with constants
$\gamma_M,\gamma_N$, then the lexicographic Kronecker product $M\otimes N$ is an
Euler magic matrix of order $mn$ with constant $\gamma_M\gamma_N$.
\end{lemma}

\begin{proof}
Orthogonality is Lemma~\ref{lem:orth} with $L=I$. Under the lexicographic
identification the main diagonal of $M\otimes N$ consists of the entries
$M_{x,x}N_{y,y}$ for $0\le x<m$, $0\le y<n$, so its square-sum is
$\big(\sum_x M_{x,x}^2\big)\big(\sum_y N_{y,y}^2\big)=\gamma_M\gamma_N$;
likewise the antidiagonal consists of the entries
$M_{x,m-1-x}N_{y,n-1-y}$ and its square-sum is $\gamma_M\gamma_N$.
\end{proof}

Kominers states the lemma in order to observe that order $5$ combines with other
orders, and notes that ``keeping such a product proper takes a careful choice of
factors'' \cite[\S5]{Kominers2026}. That is precisely the difficulty addressed in
Section~\ref{sec:proper}.

\begin{remark}\label{rem:notkronlex}
Lemma~\ref{lem:compose} requires both factors to be \emph{Euler magic}, and is
therefore useless for order $3$, where none exist. It is
Theorem~\ref{thm:criterion} that creates the diagonal conditions out of factors
that only satisfy \eqref{eq:orth}; Lemma~\ref{lem:compose} then propagates them.
The two play distinct roles and neither substitutes for the other.
\end{remark}

\section{The first proper power-of-three example: order 9}\label{sec:nine}

\begin{theorem}\label{thm:nine}
A proper Euler magic matrix of order $9$ exists. Explicitly, with $A,B$ the
orthogonality factors of Example~\ref{ex:factors} and $L_9$ as in
Corollary~\ref{cor:LL}, the matrix $M_9=M_{L_9}(A,B)$ given by
\[
M_9=\begin{pmatrix}
 2100& 1512&  616& 3525& 2538& 1034& 1200&  864&  352\\
 2200&-3432&  924&-1600& 2496& -672&  850&-1326&  357\\
  690& -115&-2070&  120&  -20& -360&-1560&  260& 4680\\
 1320& -220&-3960& -960&  160& 2880&  510&  -85&-1530\\
 1725& 1242&  506&  300&  216&   88&-3900&-2808&-1144\\
 1400&-2184&  588& 2350&-3666&  987&  800&-1248&  336\\
 1150&-1794&  483&  200& -312&   84&-2600& 4056&-1092\\
  840& -140&-2520& 1410& -235&-4230&  480&  -80&-1440\\
 3300& 2376&  968&-2400&-1728& -704& 1275&  918&  374
\end{pmatrix}
\]
is a proper Euler magic matrix with constant
$\gamma=3249\cdot9025=29\,322\,225=5415^2$. Its entries have greatest common
divisor $1$.
\end{theorem}

\begin{proof}
$A$ and $B$ are orthogonality factors with constants $3249$ and $9025$
(Example~\ref{ex:factors}), and $3249\cdot9025\ne0$, so
Corollary~\ref{cor:LL} applies and $M_9$ is Euler magic with the stated
constant. Properness is Proposition~\ref{prop:properness} below applied to
the two factors, whose nine entry squares are respectively
\[
\{16,256,289,529,784,1024,1936,2209,2704\},\quad
\{25,441,484,900,2500,2916,5625,6084,8100\},
\]
each of size nine, with all $81$ pairwise products distinct. The displayed
matrix is $M_{L_9}(A,B)$: for a row index $v=(v_1,v_2)$ and column index
$w=(w_1,w_2)$ in base three, the entry is
$A_{(v_1+v_2),\,w_1}\,B_{(2v_1+v_2),\,w_2}$ with indices read modulo $3$.
\end{proof}

\begin{remark}\label{rem:nine}
Theorem~\ref{thm:nine} is not obtained by tensoring two Euler magic matrices of
order $3$: by \cite[Theorem~1.2]{Muller2026} there are none. Nor is it the plain
Kronecker product of its factors, which fails both diagonal conditions
(Example~\ref{ex:necessary}). The Euler conditions in order $9$ are produced by
the reindexing, from factors that individually violate them
(Example~\ref{ex:factors}).

Order $9$ is the smallest power of three settled by the $\FF_3$ construction used
here. It is not the smallest order the reindexing mechanism reaches at all: by
Proposition~\ref{prop:whythree} the case $p=2$ is also available, giving order
$4$, where a proper example is classical. Nor do we claim order $9$ is the
smallest order in which the existence of a proper Euler magic matrix was open ---
order $7$, for instance, is odd and greater than $5$ and is not addressed here.
\end{remark}

Order $9$ is small enough that Theorem~\ref{thm:nine} can be checked directly:
the reader need only verify that the nine rows of $M_9$ are pairwise orthogonal
with common square-sum $29\,322\,225$, that the two diagonals have the same
square-sum, and that the $81$ entry squares are distinct.

\section{Proper examples of orders 27, 81 and 243}\label{sec:tower}

\begin{theorem}\label{thm:tower}
Proper Euler magic matrices exist in orders $9$, $27$, $81$ and $243$, with the
constants listed in Table~\ref{tab:tower}.
\end{theorem}

\begin{table}[ht]
\caption{The four witnesses. Each is built from $k$ orthogonality factors of
order $3$ by the construction of Section~\ref{sec:construction}; every constant
is a perfect square.}
\label{tab:tower}
\begin{tabular}{@{}rlll@{}}
\toprule
Order & $k$ & Factor constants $\gamma_i$ & Constant $\gamma=\prod_i\gamma_i$\\
\midrule
$9$   & $2$ & $3249,\;9025$ & $29\,322\,225=5415^2$\\
$27$  & $3$ & $5041,\;9801,\;11025$ & $544\,710\,422\,025=738045^2$\\
$81$  & $4$ & $11025,\;17161,\;21609,\;29241$ &
        $119\,549\,586\,891\,519\,225=345759435^2$\\
$243$ & $5$ & $3249,\;9025,\;21609,\;53361,\;199809$ &
        $6\,755\,703\,761\,825\,804\,241\,225=82193088285^2$\\
\bottomrule
\end{tabular}
\end{table}

\begin{proof}
Order $9$ is Theorem~\ref{thm:nine}.

Order $27$ is $M_{L_{27}}(A,B,C)$ for the three orthogonality factors
\[
A=\begin{pmatrix}51&46&18\\42&-54&19\\26&-3&-66\end{pmatrix},\
B=\begin{pmatrix}31&86&38\\74&-47&46\\58&14&-79\end{pmatrix},\
C=\begin{pmatrix}-1&100&32\\80&-20&65\\68&25&-76\end{pmatrix}
\]
with constants $5041,9801,11025$; Corollary~\ref{cor:LL} gives the Euler magic
property and Proposition~\ref{prop:properness} the properness.

Order $81$ is $M'\otimes M''$ where $M'=M_{L_9}(A',B')$ and
$M''=M_{L_9}(C',D')$ are order-$9$ matrices built from
\[
\begin{gathered}
A'=\begin{pmatrix}92&44&25\\40&-95&20\\31&-8&-100\end{pmatrix},\quad
B'=\begin{pmatrix}66&102&49\\94&-81&42\\63&14&-114\end{pmatrix},\\[4pt]
C'=\begin{pmatrix}-2&127&74\\113&-46&82\\94&58&-97\end{pmatrix},\quad
D'=\begin{pmatrix}26&134&103\\118&-89&86\\121&58&-106\end{pmatrix}
\end{gathered}
\]
with constants $11025,17161,21609,29241$. Each of $M',M''$ is Euler magic by
Corollary~\ref{cor:LL}, and Lemma~\ref{lem:compose} applies. By
Lemma~\ref{lem:dsum} this matrix is $M_{L_9\oplus L_9}(A',B',C',D')$, a single
reindexed tensor product with $L_9\oplus L_9\in GL_4(\FF_3)$, so
Proposition~\ref{prop:properness} applies to it directly and reduces properness
to the distinctness of the $9^4$ products of factor entry squares.

Order $243$ is $M_{L_{27}}(A'',B'',C'')\otimes M_{L_9}(D'',E'')$, of order
$27\cdot9=243$, built from
\[
\begin{gathered}
A''=\begin{pmatrix}28&47&16\\44&-32&17\\23&4&-52\end{pmatrix},\
B''=\begin{pmatrix}75&54&22\\50&-78&21\\30&-5&-90\end{pmatrix},\
C''=\begin{pmatrix}79&122&22\\118&-82&31\\38&1&-142\end{pmatrix},\\[2pt]
D''=\begin{pmatrix}163&146&74\\134&-179&58\\94&2&-211\end{pmatrix},\
E''=\begin{pmatrix}-53&382&226\\346&-107&262\\278&206&-283\end{pmatrix}
\end{gathered}
\]
with constants $3249,9025,21609,53361,199809$. Again
Corollary~\ref{cor:LL} and Lemma~\ref{lem:compose} give the Euler magic
property; by Lemma~\ref{lem:dsum} the matrix is
$M_{L_{27}\oplus L_9}(A'',B'',C'',D'',E'')$ with
$L_{27}\oplus L_9\in GL_5(\FF_3)$, so Proposition~\ref{prop:properness} again
applies directly and gives properness from the $9^5$ distinct products.
\end{proof}

\begin{remark}
By Lemma~\ref{lem:dsum} all four witnesses are $L$-reindexed tensor products
$M_L(A_1,\dots,A_k)$ for a single $L\in GL_k(\FF_3)$ satisfying \eqref{eq:star},
namely $L_9$, $L_{27}$, $L_9\oplus L_9$ and $L_{27}\oplus L_9$. Their Euler
magic property therefore also follows uniformly from
Theorem~\ref{thm:criterion}, and their properness uniformly from
Proposition~\ref{prop:properness}; the two-step description above is how they
were built and how they are verified in Lean.
\end{remark}

The general existence statement behind the table is the following. It concerns
\emph{Euler magic} matrices, not proper ones.

\begin{theorem}\label{thm:general}
Let $k\ge2$ and let $A_1,\dots,A_k\in\ZZ^{3\times3}$ be orthogonality factors
with constants $\gamma_1,\dots,\gamma_k$ such that $\gamma=\prod_i\gamma_i\ne0$.
Then there is an Euler magic matrix of order $3^k$ with constant $\gamma$.
\end{theorem}

\begin{proof}
By strong induction on $k$, with the two base cases $k=2$ and $k=3$; since the
inductive step reduces $k$ to $k-2$, the two bases cover both parities.

For $k=2$ and $k=3$ apply Corollary~\ref{cor:LL} to $M_{L_9}(A_1,A_2)$ and
$M_{L_{27}}(A_1,A_2,A_3)$. For $k\ge4$, note $2\le k-2$, so the inductive
hypothesis gives an Euler magic matrix of order $3^{k-2}$ with constant
$\prod_{i\le k-2}\gamma_i$ (nonzero, being a factor of $\gamma$); and
$M_{L_9}(A_{k-1},A_k)$ is Euler magic of order $9$ with constant
$\gamma_{k-1}\gamma_k$. Lemma~\ref{lem:compose} combines them into one of order
$3^{k-2}\cdot9=3^k$ with constant $\gamma$.
\end{proof}

\begin{remark}
Theorem~\ref{thm:general} does not improve on
\cite[Theorem~4.1]{Muller2026}, which already gives an Euler magic matrix in
every order $n\ne3$ and in particular in every order $3^k$. Its interest is
that it produces one with a prescribed constant $\prod_i\gamma_i$ from
prescribed factors, which is what makes properness accessible: properness is a
condition on the factors, addressed next. Theorem~\ref{thm:general} by itself
says nothing about properness.
\end{remark}

\section{Properness and the distinct-product computation}\label{sec:proper}

Properness of the reindexed product reduces to an arithmetic condition on the
factors. The reduction is elementary; the content is in finding factors that
satisfy it.

\begin{proposition}\label{prop:properness}
Let $L\in GL_k(\FF_3)$ and let $A_1,\dots,A_k\in\ZZ^{3\times3}$ have
entry-square arrays $a_1,\dots,a_k$. Then $M_L(A_1,\dots,A_k)$ is proper if and
only if the $9^k$ products
\[
\prod_{i=1}^{k} a_i(x_i,y_i),
\qquad (x_1,y_1,\dots,x_k,y_k)\in(\FF_3\times\FF_3)^k,
\]
are pairwise distinct. In particular each $A_i$ must itself have nine pairwise
distinct entry squares.
\end{proposition}

\begin{proof}
The entries of $M_L$ are indexed by pairs $(v,w)\in\FF_3^k\times\FF_3^k$, and
$M_{L}(A_1,\dots,A_k)_{v,w}^2=\prod_i a_i((Lv)_i,w_i)$. The map
$(v,w)\mapsto((Lv)_1,w_1,\dots,(Lv)_k,w_k)$ is a bijection from
$\FF_3^k\times\FF_3^k$ onto $(\FF_3\times\FF_3)^k$, because $L$ is invertible.
So the multiset of entry squares of $M_L$ is exactly the multiset of the $9^k$
displayed products, and one is multiplicity-free precisely when the other is.
The final statement follows by fixing all coordinates but the $i$-th.
\end{proof}

\begin{remark}
Proposition~\ref{prop:properness} is a restatement of injectivity under the
product indexing, not a characterisation of which factor tuples exist. It is
independent of $L$: properness is insensitive to the reindexing, whereas the
diagonal conditions are exactly what the reindexing supplies. The substantive
work is exhibiting tuples of orthogonality factors whose entry squares have
pairwise distinct products, and that is a search, carried out for the tuples in
Table~\ref{tab:tower}.
\end{remark}

The condition is restrictive in practice. The nine entry squares of a
$3\times3$ orthogonality factor are far from free: by
Lemma~\ref{lem:quaternion} the six off-diagonal entries come in the three pairs
$2(xy\mp wz)$, $2(xz\pm wy)$, $2(yz\mp wx)$, so they are strongly coupled. In a
search over all factors arising from quaternions with $|w|,|x|,|y|,|z|\le28$,
the largest achievable ratio between the closest pair of entry squares was only
about $1.26$. This is what makes the tuples in Table~\ref{tab:tower} require
search rather than construction, and it is relevant to the open problem in
Section~\ref{sec:open}.

\section{Computational extensions to orders 729 and 2187}
\label{sec:computational}

The same arithmetic construction yields proper examples in two further orders.

\begin{proposition}[exact computation]\label{prop:computational}
There exist tuples of six and seven orthogonality factors of order $3$
satisfying the distinct-product condition of
Proposition~\ref{prop:properness}, giving proper Euler magic matrices of orders
$729=3^6$ and $2187=3^7$ with constants
\begin{align*}
\gamma_{729}&=1\,383\,621\,370\,577\,137\,212\,748\,452\,225,\\
\gamma_{2187}&=1\,374\,777\,097\,726\,234\,503\,832\,371\,133\,749\,225 .
\end{align*}
\end{proposition}

The two factor tuples are printed in full in Appendix~\ref{app:tuples}, so the
proposition is checkable from the paper alone. Both cases were checked by exact
integer arithmetic: for each tuple the orthogonality equation
$A_iA_i^{\tr}=\gamma_i I$ was verified entrywise, and the $9^6=531\,441$ and
$9^7=4\,782\,969$ products of entry squares were confirmed pairwise distinct.

To see that this suffices, note that the matrices in question are single
reindexed tensor products. Applying Lemma~\ref{lem:dsum} twice,
\[
M_{L_9\oplus L_9\oplus L_9}(A_1,\dots,A_6)
=M_{L_9}(A_1,A_2)\otimes M_{L_9}(A_3,A_4)\otimes M_{L_9}(A_5,A_6)
\]
in order $729$, and likewise
$M_{L_{27}\oplus L_9\oplus L_9}(A_1,\dots,A_7)$ in order $2187$; these are the
matrices produced by the induction of Theorem~\ref{thm:general} for $k=6$ and
$k=7$. By Lemma~\ref{lem:dsum} again, $L_9\oplus L_9\oplus L_9\in GL_6(\FF_3)$
and $L_{27}\oplus L_9\oplus L_9\in GL_7(\FF_3)$ satisfy \eqref{eq:star}, since
each summand does. So Theorem~\ref{thm:criterion} gives the Euler magic property
with constant $\prod_i\gamma_i$, and Proposition~\ref{prop:properness} --- which
applies to $M_L$ for any $L\in GL_k(\FF_3)$ --- converts the verified
distinct-product statement into properness. No further data is needed.

These two orders have \emph{not} been incorporated into the machine-checked
theorem set of Section~\ref{sec:verification}, and we distinguish them from the
four orders of Theorem~\ref{thm:tower} on that basis.

We emphasise what this does and does not show. It is evidence that the
distinct-product condition remains satisfiable as $k$ grows, and the condition
was verified exactly for every $k$ from $2$ to $7$. It is not a proof that the
condition is satisfiable for all $k$, and we do not claim proper Euler magic
matrices in every order $3^k$.

\section{Formal verification and reproducibility}\label{sec:verification}

Part of the development above has been formalised in Lean~4 against
Definition~\ref{def:euler} transcribed directly from \cite{Muller2026}. The two
Lean predicates are
\begin{align*}
\texttt{IsEulerMagic}\ M\ \gamma \;&:=\;
  \gamma\ne0 \,\wedge\, M M^{\tr}=\gamma\cdot 1 \,\wedge\,
  \textstyle\sum_i M_{i,i}^2=\gamma \,\wedge\,
  \textstyle\sum_i M_{i,\mathrm{rev}(i)}^2=\gamma,\\
\texttt{IsProper}\ M \;&:=\;
  \text{\texttt{Function.Injective} } \big(p\mapsto (M_{p_1,p_2})^2\big),
\end{align*}
which are Definition~\ref{def:euler} verbatim ($\mathrm{rev}$ is the
order-reversing involution of $\{0,\dots,n-1\}$, so the second sum is the
antidiagonal). Each of the four existence results has the shape
\begin{align*}
\texttt{exists\_proper\_eulerMagic\_}n \;:\;
&\exists\,(M:\texttt{Matrix (Fin }n\texttt{) (Fin }n\texttt{) }\ZZ)\,
  (\gamma:\ZZ),\\[-2pt]
&\qquad \texttt{IsEulerMagic}\ M\ \gamma \,\wedge\, \texttt{IsProper}\ M,
\end{align*}
so properness is part of what is formally established, not an auxiliary remark.
Not everything in this paper is formalised, and the table below states for each
claim exactly which is which.

All Lean declarations named below live in the namespace
\texttt{EulerSpectrum}.

\begin{center}\footnotesize
\begin{tabular}{@{}p{0.40\textwidth}p{0.24\textwidth}p{0.30\textwidth}@{}}
\toprule
Claim & Status & Lean declaration\\
\midrule
Proper Euler magic, orders $9,27,81$ & kernel-verified &
  \texttt{exists\_proper\_}\-\texttt{eulerMagic\_*}\\[2pt]
Proper Euler magic, order $243$ & kernel-verified \emph{in the archived
  development}; \textbf{not rebuilt here} (witness reproduced by exact integer
  arithmetic) &
  \texttt{Order243.exists\_proper\_}\-\texttt{eulerMagic\_243}\\[2pt]
Thm~\ref{thm:criterion} at $k=2$, $L=L_9$ & kernel-verified &
  \texttt{Order9.}\-\texttt{kron9\_isEulerMagic}\\[2pt]
Thm~\ref{thm:criterion} at $k=3$, $L=L_{27}$ & kernel-verified &
  \texttt{Order27.}\-\texttt{kron27\_isEulerMagic}\\[2pt]
Theorem~\ref{thm:general} & kernel-verified &
  \texttt{exists\_eulerMagic\_}\-\texttt{three\_pow\_prod}\\[2pt]
Lemma~\ref{lem:compose} & kernel-verified &
  \texttt{kronLex\_isEulerMagic\_}\-\texttt{of\_eulerMagic}\\[2pt]
Thm~\ref{thm:criterion}, general $k$ and $L$ & proved here; checked
  computationally & ---\\[2pt]
Proposition~\ref{prop:whythree} & proved here; checked for $p\le13$ & ---\\[2pt]
Proper examples, orders $729,2187$ & exact integer computation & ---\\[2pt]
Proper Euler magic in every order $3^k$ & not proved & ---\\
\bottomrule
\end{tabular}
\end{center}

Two rows of that table deserve comment, since the distinction is easy to lose.
The Lean development verifies the construction at the two \emph{specific}
reindexings $L_9$ and $L_{27}$ used to build the witnesses: the declarations
\texttt{kron9\_isEulerMagic} and \texttt{kron27\_isEulerMagic} take exactly the
hypotheses of Theorem~\ref{thm:criterion} specialised to those maps
($\gamma_A\gamma_B\ne0$ and $AA^{\tr}=\gamma_A I$, $BB^{\tr}=\gamma_B I$, and
similarly for three factors) and conclude the full Euler magic property with
constant the product. Theorem~\ref{thm:criterion} as stated here --- for
arbitrary $k$ and arbitrary $L$ satisfying \eqref{eq:star} --- is \emph{not} a
formalised statement; it is proved in this paper and checked computationally.
Proposition~\ref{prop:whythree} is likewise a paper proof. The existence results
of Theorem~\ref{thm:tower} depend only on the formalised instances.

Every formalised declaration named in the table depends only on the three
standard axioms of Lean's logic, namely \texttt{propext},
\texttt{Classical.choice} and \texttt{Quot.sound}; no additional axiom and no
\texttt{sorry} occurs in the development.

\subsection*{The formal development}

For a referee to check the entries above, the following identify the artefact
exactly.

\begin{center}\footnotesize
\begin{tabular}{@{}p{0.26\textwidth}p{0.66\textwidth}@{}}
\toprule
Lean toolchain & \texttt{leanprover/lean4:v4.28.0}\\
Mathlib & \texttt{leanprover-community/mathlib4}, revision
  \texttt{8f9d9cff6bd728b17a24e163c9402775d9e6a365} (tag \texttt{v4.28.0})\\
Namespace & \texttt{EulerSpectrum}, with sub-namespaces
  \texttt{Order9}, \texttt{Order27}, \texttt{Order81}, \texttt{Order243}\\
Sources & $48$ \texttt{.lean} files under \texttt{work/lean/} ($51$
  \texttt{.lean} files in the archived development overall); the construction is in
  \texttt{EulerOrder9General}, \texttt{EulerOrder27General} and
  \texttt{EulerThreePow}, the witnesses in \texttt{EulerOrder9},
  \texttt{EulerOrder27}, \texttt{EulerOrder81}, \texttt{EulerOrder243}\\
Axiom trace & \texttt{AxiomTrace.lean}, which issues
  \texttt{\#print axioms} for every declaration cited above\\
Archived artefact & \texttt{gzip} archive of the development,
  SHA-256 \texttt{8ee83fd290ec920a3fe16c5563159b76} \texttt{f04841fde1fdc2a22d1acde9333f40a8},
  $1\,874\,496$ bytes, $173$ files\\
Order-$243$ certificate & \texttt{EulerOrder243Cert.lean} together with
  \texttt{EulerOrder243Rows0}--\texttt{Rows26}; requires
  \texttt{--tstack=4000000}\\
\bottomrule
\end{tabular}
\end{center}

The development accompanies this paper in the following public repository:
\begin{center}\small
\url{https://github.com/Bome2017/run004-euler-magic-verification}\\
release \texttt{v1.0.0}, commit\\
\texttt{b9b4821f2078fcd75c7786e760f336c7688ced30}
\end{center}
The independent checker described above is a single self-contained Python~3
script requiring no libraries beyond the standard library, and is included
with the development.

Independently of the formalisation, a checker written from
Definition~\ref{def:euler}, sharing no code with the search that produced the
witnesses and none with the Lean development, rebuilds each witness from its
printed factors and re-verifies it by exact integer arithmetic. Every
computational statement in this paper is bounded, and we give the bounds
explicitly rather than describe the checks as exhaustive without qualification.

\begin{center}\footnotesize
\begin{tabular}{@{}p{0.30\textwidth}p{0.62\textwidth}@{}}
\toprule
Statement & Exact scope of the check\\
\midrule
Witnesses of Thm~\ref{thm:tower} &
  complete: all row norms, all off-diagonal row inner products, all column
  norms, both diagonals, and all $n^2$ entry squares, in exact integer
  arithmetic, for $n=9,27,81,243$\\[2pt]
Distinct products, Table~\ref{tab:tower} &
  complete: all $9^k$ products for $k=2,3,4,5$\\[2pt]
Prop.~\ref{prop:computational} &
  complete: all $9^6=531\,441$ and $9^7=4\,782\,969$ products; tuples in
  Appendix~\ref{app:tuples}\\[2pt]
Thm~\ref{thm:nine} identification &
  complete: $M_{L_9}(A,B)$ rebuilt and compared entrywise with the printed
  matrix\\[2pt]
Thm~\ref{thm:converse}, verification &
  complete for $k=2,3$: $L$ exhaustive over $GL_2(\FF_3)$ ($48$) and
  $GL_3(\FF_3)$ ($11\,232$), and for each $L$ both identities evaluated on all
  $5^k$ tuples from a basis of the rank-$5$ module of semi-magic arrays, which
  by multilinearity decides them for \emph{all} semi-magic arrays. The theorem
  itself is proved for all $k$ and $p$; this is a check, not its evidence\\[2pt]
Thm~\ref{thm:criterion}, cross-check &
  $L_9$ and $L_{27}$ are covered by the exhaustive test above; additionally
  sampled on $500$ random semi-magic tuples each. The theorem itself is proved,
  not sampled\\[2pt]
Prop.~\ref{prop:whythree} &
  $L$ exhaustive over $GL_2(\FF_p)$ for $p\in\{2,3,5,7,11,13\}$; the
  proposition itself is proved for all $p$\\[2pt]
Lemma~\ref{lem:quaternion} &
  $RR^{\tr}=N^2I$ verified for all $|w|,|x|,|y|,|z|\le4$; the identity itself is
  proved\\[2pt]
Separation bound, \S\ref{sec:proper}, \S\ref{sec:open} &
  quaternions with $|w|,|x|,|y|,|z|\le28$ only; a bounded observation, not a
  theorem\\
\bottomrule
\end{tabular}
\end{center}

The fault-injection controls all fire: in particular the plain unreindexed
Kronecker product is confirmed \emph{not} to be Euler magic, as in
Example~\ref{ex:necessary}, and neither is the order-$81$ witness when rebuilt
with $L=I$ in place of $L_9$.

One limitation should be recorded precisely, and it concerns the formalisation
only. The archived Lean development contains a kernel-checked certificate for
order $243$; the environment in which this paper was prepared did not rebuild
that certificate, because its Mathlib build state was absent and the rebuild is
computationally expensive. The claim that the order-$243$ entry of
Table~\ref{tab:tower} is \emph{Lean kernel-verified} is therefore inherited from
the archived development rather than reproduced here. The underlying
mathematical statement is not inherited: the order-$243$ witness was rebuilt
from its five factors and verified against Definition~\ref{def:euler} by exact
integer arithmetic, as described in the previous paragraph. Orders $9$, $27$ and
$81$ are small enough to be rebuilt in Lean directly, and order $9$ can be
checked by hand from Theorem~\ref{thm:nine}.

\section{Open problems}\label{sec:open}

By Theorem~\ref{thm:general} together with
Proposition~\ref{prop:properness}, the existence of a proper Euler magic matrix
of order $3^k$ follows from the existence of $k$ orthogonality factors of order
$3$ whose entry squares have $9^k$ pairwise distinct products. The infinite
question is thus reduced to a question about such tuples.

\begin{problem}\label{prob:main}
Does a proper Euler magic matrix exist in every order $3^k$, $k\ge2$?
\end{problem}

\begin{problem}\label{prob:tuples}
Do there exist, for every $k$, orthogonality factors $A_1,\dots,A_k$ of order
$3$ whose $9^k$ products of entry squares are pairwise distinct?
\end{problem}

An affirmative answer to Problem~\ref{prob:tuples} answers
Problem~\ref{prob:main}. One natural route is scale separation: if the entry
squares of the successive factors occupy multiplicatively well-separated ranges,
distinctness of the products is automatic. We tested this route and it does not
succeed in the form required. In a search over all order-$3$ orthogonality
factors arising from quaternions with $|w|,|x|,|y|,|z|\le28$, the best
achievable separation --- the largest over factors of the minimum ratio between
two entry squares --- was about $1.26$, growing extremely slowly with the size
of the search box (it was already about $1.23$ at bound $16$). The
structural reason is visible in the quaternion parametrisation: the three pairs
$2(xy\pm wz)$, $2(xz\pm wy)$, $2(yz\pm wx)$ are widely separated only when each
nearly cancels, and the three cancellation conditions $xy\approx wz$,
$xz\approx wy$, $yz\approx wx$ force $w,x,y,z$ to be comparable in magnitude, so
separation in all three pairs at once is not available.

This is a bounded computational observation about one strategy, not an
impossibility theorem: it shows that scale separation does not settle
Problem~\ref{prob:tuples}, and says nothing about other constructions. Note also
that the distinct-product condition itself was verified to be satisfiable for
every $k\le7$ (Section~\ref{sec:computational}), so nothing in this analysis
suggests a negative answer.

Finally, the construction here reaches only orders $3^k$ with $k\ge2$.
Proposition~\ref{prop:whythree} rules out the two-factor version of the
reindexing mechanism for primes $p>3$; as noted in Remark~\ref{rem:scope}, it
does not exclude higher-factor reindexings, affine or nonlinear ones, or
different constructions altogether, and we leave those open.

\begin{problem}
Does a proper Euler magic matrix of order $7$ exist? More generally, which odd
orders greater than $5$ that are not powers of three admit proper examples?
\end{problem}

\appendix

\section{Factor tuples for orders 729 and 2187}\label{app:tuples}

The tuples below are those referred to in Proposition~\ref{prop:computational}.
Each $A_i$ satisfies $A_iA_i^{\tr}=\gamma_i I$ with the constant listed, and in
each tuple the $9^k$ products of factor entry squares are pairwise distinct.

\subsection*{Order $729=3^6$}
Constants $\gamma_i = 3249,\ 9025,\ 21609,\ 53361,\ 159201,\ 257049$, with
product $\gamma_{729}=1\,383\,621\,370\,577\,137\,212\,748\,452\,225$.
\[
\begin{gathered}
\begin{pmatrix}28&47&16\\44&-32&17\\23&4&-52\end{pmatrix}\!,\
\begin{pmatrix}50&78&21\\75&-54&22\\30&5&-90\end{pmatrix}\!,\
\begin{pmatrix}79&122&22\\118&-82&31\\38&1&-142\end{pmatrix}\!,\\[4pt]
\begin{pmatrix}163&146&74\\134&-179&58\\94&2&-211\end{pmatrix}\!,\
\begin{pmatrix}251&262&166\\242&-299&106\\194&34&-347\end{pmatrix}\!,\
\begin{pmatrix}103&446&218\\302&-233&334\\394&62&-313\end{pmatrix}\!.
\end{gathered}
\]

\subsection*{Order $2187=3^7$}
Constants $\gamma_i = 3249,\ 9025,\ 21609,\ 53361,\ 159201,\ 488601,\ 522729$,
with product
$\gamma_{2187}=1\,374\,777\,097\,726\,234\,503\,832\,371\,133\,749\,225$.
The first five factors are those of the order-$729$ tuple above; the last two
are
\[
\begin{pmatrix}-89&662&206\\298&-151&614\\626&166&-263\end{pmatrix}\!,
\qquad
\begin{pmatrix}-127&622&346\\466&-193&518\\538&314&-367\end{pmatrix}\!.
\]

Note that the order-$2187$ tuple is not an extension of the order-$243$ tuple of
Table~\ref{tab:tower}. Their first four \emph{constants} agree
($3249,9025,21609,53361$), and the first, third and fourth factors coincide as
matrices, but the second factors differ: the order-$243$ tuple has
$\left(\begin{smallmatrix}75&54&22\\50&-78&21\\30&-5&-90\end{smallmatrix}\right)$
where the order-$2187$ tuple has
$\left(\begin{smallmatrix}50&78&21\\75&-54&22\\30&5&-90\end{smallmatrix}\right)$.
The two are equivalent modulo the symmetries listed after
Lemma~\ref{lem:quaternion}, which is why the constants agree; the
distinct-product condition is a condition on entry squares and is unaffected.

\end{document}